\documentclass[10pt]{article}
\usepackage{amsmath, amsthm}

\newtheoremstyle{theorem}
{10pt} 
{10pt} 
{\sl} 
{\parindent} 
{\bf} 
{.} 
{ } 
{} 

\theoremstyle{theorem}
\newtheorem{theorem}{Theorem}
\newtheorem{corollary}[theorem]{Corollary}

\newtheoremstyle{defi}
{10pt} 
{10pt} 
{\rm} 
{\parindent} 
{\bf} 
{. } 
{ } 
{} 

\theoremstyle{defi}

\newtheorem{definition}[theorem]{Definition}

\begin{document}
\title {Riesz type theorem in locally convex vector spaces}
\author {Miloslav Ducho\v n\\[2pt]
Mathematical Institute, Slovak Academy of Sciences,\\
 \v Stef\'anikova 49,
SK-814 73 Bratislava,  Slovakia\\
miloslav.duchon@mat.savba.sk } \maketitle
\begin{abstract}
The present paper is concerned with some representatons of linear mappings of continuous functions into locally convex vector spaces, namely
{Theorem} : If $X$ is a complete Hausdorff
locally convex vector space, then a general form of weakly compact mapping
   $ T:C{[a,b]}\to X $
is of the form
$Tg=\int_a^bg(t)dx(t)$,
where the function $x(\cdot):[a,b] \to X$ has a weakly compact semivariation on $[a,b]$.
This theorem is a generalization of the result  from Banach spaces to locally convex vector spaces.

$$ 
 $$ 




\thanks{Supported by grant agency Vega, N. 2/0212/10}.

\end{abstract}

\section{Introduction}

The set $\mathcal{WCS}$ of functions $x(\cdot)$ of weakly compact semivariation on the real line $R$ to a complete complex locally convex vector space $X$ derives its importance from the fact that each function $x(\cdot)$ in $\mathcal{WCS}$ is associated in a natural way with a vector-valued measure on the ring $\mathcal B$ of bounded Borel sets, and conversely. The relationship resembles that between ordinary real and complex functions of bounded variation and real and signed measures defined on $\mathcal B$. A precise statement will be given in sect 5.

The extension to locally convex vector space is very useful because it makes possible to include also many important spaces, which are not normable, e. g., nuclear locally convex vector spaces often appearing in applications.

\section {Functions with bounded semivariation}

Let $X$ be a locally convex vector space and $x(\cdot)$ a function on an interval $[a,b]$ with
values in $X$. We denote by $X'$ the topological dual of $X$.

We recall now some definitions and facts, see \cite{hi}.

\begin{definition}
We will say that $x(\cdot)$ has {\it bounded semi-variation} if the set
$$V_{x(\cdot)}=\{\sum_{j=1}^{k}\alpha_{j}(x(t_{j})-x(t_{j-1}))\mid
k\ge 1, a=t_{0}<t_{1}<\ldots<t_{k}=b,|\alpha_{j}|\leq
1\}$$is  bounded  in $X$.

Equivalently, $x(\cdot)$ has {\it bounded semi-variation} if and only if $p(V_{x(\cdot)})$ is a bounded set of real numbers for each $p\in Q$, where $Q$ is the set of seminorms determining the topology of $X$.
\end{definition}
\medskip
The following  two theorems are proved, e. g., in \cite{dd}.
\begin{theorem}
    A function $x(\cdot)$ has bounded semi-variation if and only if elements
    of the set $\{x'\circ x\mid x'\in X'\} $ have uniformly bounded variations.
\end{theorem}

Recall that if $A \subset X$ is bounded set then $p_A'(x')= \sup\{|(x,x')|:x\in A\}$ is a seminorm on $X'$ (see \cite{ro}, p.45).
\begin{theorem}
The function $x(\cdot):[a,b]\to X$ has bounded semivariation on $[a,b]$ if and only if the elements of the set $\{x'\circ x\mid x'\in X', p'_{E}(x')\leq 1\}$ have uniformly bounded  variations on $[a,b]$.
\end{theorem}

\section {Functions with weakly compact semivariation}
Let us consider a function $x(\cdot): R\to X$.
Denote, for $[a,b]$,
$$
E(a,b)= \{\sum_{i=1}^n[x(t_{2i}- x(t_{2i-1})], n\ge 1, -\infty < a \le t_1< t_2<\dots<t_{2n}\le b<\infty \},
$$
We shall say that the function $x(\cdot):R\to X$ has a weakly compact semivariation if the set $E(a,b)$ is relatively weakly compact.
It can be proved , Edw  Th.[5], or Stud. [11]
\begin{theorem}
If $X$ is a sequential complete locally convex vector space and $x(.):R\to X$ has on every interval $[a,b]$ weakly compact semivariation, then the limits of the function $x(\cdot)$ from the right and from the left, $x(\pm 0)$, exist in every point $t$ and if $X$ is also metrisable, then the function $x(.)$ is continuous for all $t$ except of the most countable set.
\end{theorem}

In the next chapter, there is presented the existence of the unique correspondence of the functions continuous from the right with weakly compact semivariation  on $[a,b]$ with vector measures on borelian sets in $[a,b]$.

\section { Riemann-Stieltjes integral in locally  convex space}

 The notion of Riemann-Stieltjes integral can be extended to vector valued functions with values in   locally convex vector space in two directions.[ For the case of Banach space we refer to Hille \cite {hi})]. The values in locally convex vector space  $X$ can take either the integrated function  or the integrand.

Let further $x(\cdot):[a,b] \to X$ and $g(\cdot\cdot$ be the scalar-valued function defined on $[a,b]$. For the division of the interval $[a,b]$, $t_0=a \le t_1\le \dots\le t_n=b$ together with system of points $s_i, \, t_{i-1} \le s_i \le t_i,$ we put $ d= \max_{i} {|t_i-t_{i-1}|}$.

Similarly as for scalar functions we may define two kinds of integrals.
\begin{definition}
Let
$$
S_d(x,g)= \sum_{i=1}^n x(s_i)[g(t_i)-g(t_{i-1})].
$$
If there exists $\lim_{d\to 0} S_d$ as an element of the space $X$, we denote it
$$
\int_a^b x(t)dg(t)
$$
\end{definition}

\begin{definition}
Let
$$
s_d(g,x)  = \sum_{i=1}^n g(s_i)[x(t_i)-x(t_{i-1})]
$$
If there exists $\lim_{d\to 0} s_d$ as an element of the space $X$, we denote it
$$
\int_a^b g(t)dx(t)
$$
\end{definition}
In an integral calculus the important place has the theorem on integration "per partes". In this way defined integral has the property that for it is possible to generalize a theorem on integration per partes.
\begin{theorem}
If there exists either integral
$$
\int_a^b x(t)dg(t), \int_a^b g(t)dx(t),
$$
then there exists also the second one and it holds
$$
\int_a^b x(t)dg(t) = x(b)g(b) - x(a)g(a) - \int_a^b g(t)dx(t).
$$
\end{theorem}
\begin{proof} (cf. Hille  \cite {hi} ) It is clear, after some elementary steps, that it holds
$$
\sum_{i=1}^n x(s_i)[g(t_i)-g(t_{i-1})]= x(b)g(b)- x(a)g(a) -\sum_{i=0}^n[g(t_i)[x(s_{i+1})-x(s_i)],
$$
where $s_0 = a, s_{n+1}=b$, since $s_0=a\le s_1\le s_2\le \dots \le s_{n+1}=b$ is also a division of the interval $[a,b]$
with the property $s_i \le t_i \le s_{i+1}$ and $\max_{i} |s_{i+1}-s_i |\le 2d$. By limiting process we obtain the validity of the theorem.
\end{proof}

\begin{theorem}
Let either
$ 1)\, \, x(\cdot):[a,b]\to X$, where $X$ is sequentially complete locally convex vector space, be continuous on $[a,b]$ and $g(\cdot)$ be scalar valued function on $[a,b]$ with bounded variation or

$ 2) \,\, x(\cdot):[a,b]\to X$, where $X$ is sequentially complete locally convex vector space, have bounded $p$-semivariation on $[a,b]$ and $g(\cdot)$ be a continuous scalar valued function  on $[a,b]$.

Then the integrals
$$
\int_a^b x(t)dg(t) ,\, \int_a^b g(t)dx(t).
$$
exist in the topology of the space $X$.
\end{theorem}
\begin{proof}
1) Let $x(\cdot):[a,b]\to X$ be a continuous function and $g(\cdot)$ be a scalar valued function on $[a,b]$ with bounded variation. Since $x(\cdot)$ is uniformly continuous, for all $p\in P$, for arbitrary $\epsilon > 0$ there exists $\delta > 0$ such that for all $s_1, s_2$ such that $|s_1-s_2| < \delta$ there holds
$$
p[x(s_1)-x(s_2)] < \epsilon.
$$

Considering divisions such that $d_1, d_2 < {\delta}/2$ it holds
$$
p(S_{d_1} - S_{d_2})\le 2 \epsilon\, var[g(\cdot),[a,b]],
$$
which implies the existence of the integral
$$
\int_a^bx(t)dg(t)
$$
in the completion of the space $X$.
But since the set of elements
$$
\sum_{i=1}^n x(s_i)[g(t_i) - g(t_{i-1})]
$$
is bounded, the integral
$$
\int_a^b x(t)dg(t)
$$
belongs to the quasicompletion of the space $X$, too. But
since the value of the integral does not depend of the choice of generalized
sequence of partial sums, we may express the integral as the limit of a sequence
of partial sums and since $X$ is sequentially complete, the integral belongs
to $X$, too.

2) Let $g(\cdot)$ be continuous on $[a, b]$ and let $x(\cdot)$ have the bounded $p$-semivariation
on $[a, b]$ by a number $K_p (a, b) < \infty$, for all $p \in P$. Since the function $g(\cdot)$ is
uniformly continuous, for arbitrary $\epsilon > 0 $, there exists $\delta > 0 $ such that for all $s_1, s_2$
for which $|s_1 - s_2| < \delta $ there holds
$$
|g(s_1) - g(s_2)| < \epsilon .
$$
Considering divisions such that $d_1, d_2 < {\delta}/2$
and arbitrary $p\in P$ it holds
$$
p(s_{d_1} - s_{d_2})\le 8 \epsilon\, K_p (a,b).
$$
From the similar reasons as in the first part of the proof it follows that there exists the integral
$$
\int_a^b g(t)dx(t)
$$
in $X$.
\end{proof}

\begin{corollary}
Since in the assumptions of Theorem 3 the space $X$ is sequentially complete, the Theorem 3 is valid in the part 2, if $x(\cdot)$ has weakly compact semivariation on $[a,b]$.
\end{corollary}

\section{ Weakly compact mappings in locally convex vector spaces}

In this section we study a general form of weakly compact mapping of continuous functions into a locally convex vector space  by means of above defined integral.

\begin{definition}
A mapping $T$ from locally convex vector space $X$ into locally convex vector space $Y$ is called weakly compact if there exists a neighborhood $U$ of zero in $X$ and weakly compact set $K$ in $Y$ such that there holds $T(U) \subset K$.
\end{definition}

\begin{theorem} (Sirvint, Edwards). If X is a complete locally convex vector space, then a general form of weakly compact linear mapping
\begin{equation}
    T:C{[a,b]}\to X
\end{equation}
is given by
\begin{equation}
Tg=\int_a^bg(t)dx(t),
\end{equation}
where the function $x(\cdot)$ has a weakly compact semivariation on $[a,b]$.
\end{theorem}

This theorem is a generalization, of the results proved for Banach spaces by Sirvint (\cite
{si}) and also by Edwards (\cite {ed}), to locally convex spaces by means of results of (Tweddle \cite {tw}, Lewis, \cite {le}, Deb \cite {de}).
\begin{proof}
The existence of the integral follows from the following. Let $x(\cdot)$ have a weakly compact semivariation on $[a,b]$.
Consider first the case that the function $g(\cdot)$ is real valued and that $0\le g(t)\le 1$, where $a\le t\le b$. Then
\begin{equation}
y = \sum_{i=1}^n g(\tau_i)[x(t_{i+1}) - x(t_i)] = \sum_{r=1}^n g(\tau_{i_r} )[x(t_{i_{r}+1}) - x(t_{i_r})],
\end{equation}
where
\begin{equation}
0\le g(\tau_{i_{1}})\le  g(\tau_{i_{2}})\le\dots\le g(\tau_{i_{n}})\le 1.
\end{equation}
Moreover there holds
\begin{equation}
y = g(\tau_{i{_1}})\sum_{r=1}^n[x(t_{{i_r}+1})- x(t_{i_r})] +
\sum_{s=2}^n\{ [g(\tau_{i_s} ) - g(\tau_{i_{s-1}} )]\sum_{r=s}^n
[x(t_{i_r+1}) - x(t_{i_r})]\},
\end{equation}
Since
\begin{equation}
g(\tau_{i_1})\ge 0, g(\tau_{i_s} - g(\tau_{i_{s-1}}) \ge 0, s=2,3,\dots,n
\end{equation}
and
\begin{equation}
g(\tau_{i_1}) + \sum_{s=2}^n \{g(\tau_{i_s}) - g(\tau_{i_{s-1}})\} =
g(\tau_{i_{n}})\le 1,
\end{equation}
it is evident that $y$ and so  also $Tg$ as limit belongs to the closure of absolute convex hull, $F$, of the set
$E(a,b)\cup \{0\}$.
But $E(a,b)$, and therefore $E(a,b)\cup \{0\}$ is conditionally weakly compact, and hence , by
theorem of Krein in locally convex spaces (Rob \cite {ro} p. 18), also Tweddle \cite{tw})  $F$ is weakly compact. Since
\begin{equation}
T\{g:\|g\|\le 1\}\subset \{\sum_{r=0}^3 i^rx_r:x_r\in F\}
\end{equation}
it follows, by application of Eberlein's theorem in lcs (Rob \cite {ro} or Tweddle \cite {tw}) that $T$ is a weakly compact mapping. Indeed, take now a function $g(\cdot)$ such that $\|g|\ \le 1 $. The function $g(\cdot)$ can be written in the form
$$
g(t)=g_0(t)- g_2(t) + i[g_1(t) - g_3(t)], \,t\in [a,b],
$$
where $g_i(t)$ are nonnegative real valued functions such that
$$
0\le g_i(t) \le 1, \, i=0,1,2,3.
$$
Since
$$
T(g)= \sum_{r=0}^3 i^r Tg_i,
$$
where $Tg_i,\, i=0,1,2,3$, belongs to the weakly compact set $F$, the mapping $T$  maps the unit sphere in $C[a,b]$  into the weakly compact set in $X$, i.e.,the mapping $T$ is weakly  compact.

Conversely, if $T: C[a,b] \to  X $ is weakly compact mapping, then there exists a vector measure $m:\mathcal{B}\cap [a,b]\to X$ such that for $x'\in X' ,
$ there holds
\begin{equation}
x'(Tg)=\int_a^b g(t)x'm(dt), x'\in X',
\end{equation}
Lewis (\cite {le}), Tweddle (\cite {tw}).

If we put
\begin{equation}
y(t)=m([a,t]), a <t \le b,
\end{equation}
$$
y(a)=0,
$$
then $y(\cdot) \in \mathcal{WCS}(a,b)$ and
$$
x'(Tg)= \int_{[a,b]} g(t) x'm(dt)= \int_a^b  g(t)dx'y(t)=
$$
$$
=x'\{\int_a^b g(t)dy(t)\}, x'\in X',
$$
so that
$$
Tg=\int_a^b g(t)dy(t)
$$
as required.
\end{proof}

Denote by $\mathcal {WCS}_0$ the class of all functions in $\mathcal {WCS}$ which are (strongly) continuous to the right.

If now $x(\cdot) \in \mathcal {WCS}_0$ is given, then the mapping $T:C(a,b)\to X$ defined by
$$
Tg=\int_a^b g(t) dx(t)
$$
is weakly compact. If the associated $m:\mathcal{B} \cap {[a,b]}\to X$ and $y:[a,b] \to X$ are constructed as in the foregoing proof (see (9 ) and (10)), then we have
$$
\int_a^b g(t)dx(t)= \int_a^b g(t)dy(t), \, g\in C(a,b),
$$
and then, by a usual way,
$$
x(t)-x(a)=y(t)-y(a), \, a\le t\le b.
$$
Hence $y(a)=y(a+0)$ and therefore $m(a)=0$, so that
$$
y(t)=m([a,t])=m([a,t])-m((a))=m((a,t]),a<t \le b.
$$

If we now denote by $m_0$ the restriction of $m$ to $\mathcal {B}\cap(a,b]$, then we have
$$
m_0(E)= y(d)-y(c)=x(d)-x(c),
$$
whenever $E=(c,d]\subset {(a,b]}.$
We can thus construct $m_0$ for $- \infty < a < b < \infty $ and it is now possible as
before to complete the proof of the following
\begin{corollary}
If $x(\cdot): R \to X$ is a function of the class $WCS_0$, where $X$ is an
arbitrary quasicomplete locally convex vector space, then there exists a unique
vector-valued measure $ m: \mathcal B \to X$ satisfying
$$
x(d)- x(c) = m(E),\, E =(c, d].
$$
\end{corollary}


\begin{thebibliography}{99}
\bibitem{bds} R. G. Bartle, N. Dunford. J. Schwartz, Weak compactness and vector measures,{\it Canadian J. of Math.} {\bf 7} (1955),289-305.
\bibitem{de} C.Debi\`eve, Int\'{e}gration par rapport \`{a} une mesure vectorielle, {\it Ann. de
la Societ\'e Scientif. de Bruxelles}, {\bf 11} (1973), 165--185.
\bibitem{di} N.Dinculeanu,{\it Vector Measures},VEB Deutscher Verlag der Wissenschaften, Berlin
(1966)
\bibitem{dd} M. Ducho\v n, C. Debi\`eve,Functions with bounded variation in locally convex space, 
{\it Tatra Mt. Math. Publ.} {\bf 49} (2011), 89--98.

\bibitem{ds}  N.Dunford, J.T. Schwartz, {\it Linear Operators, Part I}, Interscience, New
York, USA, (1966), 
\bibitem{ed} D. A. Edwards,  On the continuity properties of
functions satisfying a condition of Sirvint's, {\it  Quart. J. Math. Oxford}
{\bf 2} (1957), 58--67.
\bibitem{hi} E. Hille, {\it Methods in Classical and Functional Analysis}, Adison-Wesley
Publishing Company, Massachusets, USA, (1972).
\bibitem{le} D. R. Lewis, Integration with respect to vector measures,{\it Pac. J.
Math.}, {\bf 33} (1970), 157--165.
\bibitem{ro} A. P. Robertson, W. J. Robertson, {\it Topological Vector Spaces}, Cambridge, (1963).
\bibitem{si} G.Sirvint,Weak compactness in Banach spaces,{\it Studia Math.}{\bf 11}(1950),71--94.
\bibitem{st} A. Studen\'a, {\it On Some Properties of Continuity of Vector Valued Functions, PhD Dissertation}, Bratislava, (1985), (in Slovak)
\bibitem{tw} I. Tweddle,  Weak compactness in locally convex spaces, {\it Glasgow Math. J.}, {\bf 9} (1968), 23--127.

\end{thebibliography}
\end{document}